\documentclass[12pt]{amsart}
\usepackage{amsmath,amsthm,amssymb}
\usepackage{young}
\usepackage[enableskew]{youngtab}
\usepackage[all]{xy}

\newcommand{\nexteq}{\displaybreak[0]\\ &=}
\newcommand{\nnext}{\displaybreak[0]\\ &}

\newtheorem{lem}{Lemma}

\newtheorem{thm}[lem]{Theorem}

\theoremstyle{definition}

\newtheorem{exam}[lem]{Example}
\newtheorem{alg}{Algorithm}

\DeclareMathOperator{\SSYT}{SSYT}

\newcommand{\la}{\leftarrow}
\newcommand{\fS}{\mathfrak{S}}

\newcommand{\cR}{\mathcal{R}}
\newcommand{\cL}{\mathcal{L}}

\title[Give a corrected bijective proof of Vershik's relations]{A correction to  the paper `A new approach to the representation theory of the symmetric groups, III'}
\author{Minwon Na}
\address{Research Center for Pure and Applied Mathematics,
Graduate School of Information Sciences,
Tohoku University, Sendai 980--8579, Japan}
\email{minwon@ims.is.tohoku.ac.jp}

\date{February 13, 2017}

\begin{document}
\maketitle
\begin{abstract}
The aim of this paper is to give a corrected bijective proof of Vershik's relations for the Kostka numbers.
Our proof uses insertion and reverse insertion algorithms, as in the combinatorial proof of 
the Pieri rule.
\end{abstract}


\section{Introduction}\label{vsec.1}

The aim of this paper is a corrected bijective proof of \cite[Theorem 4]{V}.
First of all, we recall some notations and definitions.
We write
$\lambda\vDash n$ if $\lambda$ is a composition of $n$, that is, a sequence $\lambda=(\lambda_1,\lambda_2,\dots,\lambda_h)$ of nonnegative
integers such that $|\lambda|=\sum_{i=1}^h\lambda_i=n$. 
In particular, 
if a sequence $\lambda$ is non-increasing and $\lambda_i>0$ 
for all $1\leq i\leq h$, then we write $\lambda\vdash n$ and say that 
$\lambda$ is a partition of $n$. 
We denote by 
$\lambda^{(i)}$ the composition of $n-1$ defined by $\lambda^{(i)}_i=\lambda_i-1$, and $\lambda^{(i)}_j=\lambda_j$ otherwise. 
For
$\lambda=(\lambda_1,\dots,\lambda_h)\vdash n$ and $\gamma\vdash
n-1$, we write $\gamma\preceq\lambda$
if $\gamma_i\leq\lambda_i$ for all $i$ with
$1\leq i\leq h$

Vershik has introduced a relation for the Kostka numbers (see \cite[p.143, Theorem 3.6.13]{C} and \cite[Theorem 4]{V}): for any $\lambda\vdash n$ and $\rho\vdash n-1$, we have
\begin{equation}\label{eq:1}
\sum_{\substack{\mu\vdash n\\ \mu\succeq\rho}}K(\mu,\lambda)
=\sum_{\substack{\gamma\vdash n-1\\
\gamma\preceq\lambda}}c(\lambda,\gamma)K(\rho,\gamma),
\end{equation}
where $c(\lambda,\gamma)$ the number of ways to obtain the partition $\lambda\vdash n$ 
from partition $\gamma\vdash n-1$.
This relation arises from restricting a permutation representation of 
the symmetric group $\fS_n$ to $\fS_{n-1}$ and then applying Young's rule to 
both sides. 
%
Vershik \cite[Theorem 4]{V} claims to give a bijective proof, but it 
is poorly explained and incorrect (see Example~\ref{vexam.1} below). 
The purpose of this paper is to give a bijective proof of (\ref{eq:1}) using insertion and reverse insertion algorithms. We remark that this bijection is obtained by the restiction of a bijection giving 
the Pieri rule (see \cite[p.402, 10.65]{Ni}).

This paper is organized as follows. 
After giving preliminaries in Section~\ref{vsec.2}, we 
prove (\ref{eq:1}) in Section~\ref{vsec.3}. 

\section{Preliminaries}\label{vsec.2}

Throughout this paper, let $h\geq1$, $x\geq 1$ and $n\geq1$ be integers. 
We denote by 
$D_{\mu}$ the Young diagram of
$\mu$. 
The rows and the columns are numbered from top to bottom and from left to 
right, like the rows and the columns of a matrix, respectively. 
A semistandard
Young tableau (SSYT) of shape $\mu$ and weight, or content, $\lambda=(\lambda_{1},\ldots,\lambda_{h})$ is a
filling of the Young diagram $D_{\mu}$ with the numbers
$1,2,\ldots,h$ in such a way that $i$ occupies $\lambda_i$ boxes, for $i=1,2,\ldots,h$,
and the numbers are strictly increasing down the columns and weakly increasing along the rows.
We denote by $\SSYT(\mu,\lambda)$ the set of all semistandard tableaux of
shape $\mu$ and weight $\lambda$. 
The Kostka number $K(\mu,\lambda)$ is defined to be the 
cardinality of $\SSYT(\mu,\lambda)$.

Let
$\mu\vdash n$ and 
$\lambda=(\lambda_1,\lambda_2,\ldots,\lambda_h)\vDash n$, and 
$T\in\SSYT(\mu,\lambda)$. 
First of all, we need a fundamental combinatorial algorithm on tableaux 
called \emph{row-insertion}, or \emph{bumping} (see \cite[Chapter 1]{F}).
We define an insertion tableau, denoted $T\la x$, 
by the following procedure.
 
 \begin{alg}
\textbf{Input}: Let $T\in\SSYT(\mu,\lambda)$ and 
$x$ be a positive integer.\\
\textbf{Output}: $T\la x$\\
\textbf{Initialization}: $S:=T$, $y:=x$ and $i:=1$\\
\textbf{while} $\mid\{j\mid y<S(i,j)\}\mid>0$ \textbf{do}\\
\begin{tabular}{c|l}
&\;$z:=\min\{j\mid y<S(i,j)\}$.\\
&\;$x':=S(i,z)$.\\
&\;\textbf{if} $(p,q)=(i,z)$ \textbf{then} $U(p,q):=y$\\
&\;\textbf{else}\\
&\;\;$U(p,q):=S(p,q)$\\
&\;\textbf{end if}\\
&\;$S\leftarrow U$, $y\leftarrow x'$ and $i\leftarrow i+1$.\\
\end{tabular}\\
\textbf{end}\\
$(T\la x)(i,\mu_i+1):=y$\\
Otherwise, $(T\la x)(p,q):=S(p,q)$\\
Output $T\la x$.
\end{alg}

This algorithm is invertible. Given a partition $\mu\vdash n$ 
and insertion tableau $T\la x$, there is the unique box of $T\la x$ not in 
$D_{\mu}$. 
From this box, 
we can construct the reverse
insertion algorithm, so we can recover the original tableau $T$.

\section{Vershik's relations for the Kostka numbers}\label{vsec.3}

Let $[h]=\{1,2,\ldots, h\}$, and let
$\SSYT_{[h]}(\mu)$ be the set of all SSYT's of shape $\mu$ and taking values
in $[h]$. For a partition
$\rho\vdash n-1$, Loehr \cite[p.399, 10.60]{Ni}
shows that insertion $I$ and reverse insertion $R$
give mutually inverse bijections 
\begin{align*}
I&:\SSYT_{[h]}(\rho)\times[h]\to
\bigcup_{\substack{\mu\vdash n\\ \mu\succeq\rho}}\SSYT_{[h]}(\mu),\\
R&:\bigcup_{\substack{\mu\vdash n\\ \mu\succeq\rho}}\SSYT_{[h]}(\mu)
\to\SSYT_{[h]}(\rho)\times[h].
\end{align*}
given by $I(T, x)=T\la x$ and $R(S)$ is the result of applying reverse insertion to $S$ starting at the unique box of $S$ not in $\rho$. 

\begin{thm}\label{vthm.1}
Let $\rho\vdash n-1$ and $\lambda=(\lambda_1,\dots,\lambda_h)\vDash n$ 
and,
set
\begin{align*}
\cR'&=\bigcup_{1\leq x\leq h}(\SSYT(\rho,\lambda^{(x)})\times\{x\}),\\
\cL&=\bigcup_{\substack{\mu\vdash n\\ \mu\succeq\rho}}\SSYT(\mu,\lambda).
\end{align*}
Then $I|_{\cR'}$ and $R|_{\cL}$ give mutually inverse 
bijections between $\cR'$ and $\cL$.
\end{thm}

\begin{proof}
It is obvious that
\begin{align*}
\cR'&\subset\SSYT_{[h]}(\rho)\times[h],\\
\cL&\subset\bigcup_{\substack{\mu\vdash n\\ \mu\succeq\rho}}\SSYT_{[h]}(\mu).
\end{align*}
For each $x\in [h]$,
we have
\[I(\SSYT(\rho,\lambda^{(x)})\times\{x\})\subset\cL\]
by the definition of insertion. This implies $I(\cR')\subset\cL$.
Conversely, for each $\mu\vdash n$ with $\mu\succeq\rho$, there
exists $\ell$ such that $D_\mu=D_\rho\cup\{(\ell,\mu_\ell)\}$.
Applying reverse insertion at $(\ell,\mu_\ell)$ for each
tableau in $\SSYT(\mu,\lambda)$, we find
$R(\SSYT(\mu,\lambda))\subset\cR'$. This implies
$R(\cL)\subset\cR'$. Since $I$ and $R$ are mutually inverse
bijections, we have
\begin{align*}
\cR'&=RI(\cR')\subset R(\cL)
,\\
\cL&=IR(\cL)\subset I(\cR')
.
\end{align*}
Therefore, $I(\cR')=\cL$ and $R(\cL)=\cR'$.
\end{proof}


From Theorem~\ref{vthm.1}, we can prove (\ref{eq:1}) as follows: 
\begin{align*}
\sum_{\substack{\mu\vdash n\\ \mu\succeq\rho}}K(\mu,\lambda)&=
\sum_{\substack{\mu\vdash n\\
\mu\succeq\rho}}|\SSYT(\mu,\lambda)|
\nexteq
\sum_{1\leq x\leq h}|\SSYT(\rho,\lambda^{(x)})|&&\text{(by
Theorem~\ref{vthm.1})}\nexteq
\sum_{\substack{\gamma\vdash n-1\\
\gamma\preceq\lambda}} \sum_{\substack{1\leq x\leq h\\
\widetilde{\lambda^{(x)}}=\gamma}} |\SSYT(\rho,\lambda^{(x)})|
\nexteq \sum_{\substack{\gamma\vdash n-1\\ \gamma\preceq\lambda}}
\sum_{\substack{1\leq x\leq h\\ \widetilde{\lambda^{(x)}}=\gamma}}
|\SSYT(\rho,\gamma)| &&\text{(by \cite[Lemma 3.7.1]{C}})\nexteq
\sum_{\substack{\gamma\vdash n-1\\ \gamma\preceq\lambda}}
c(\lambda,\gamma)K(\rho,\gamma).
\end{align*}



Finally, we compare Vershik's claimed bijection with ours. 
Vershik calls a tableau in $\SSYT(\mu,\lambda)$ 
a $\mu$-tableau, and a tableau in $\SSYT(\rho,\lambda^{(x)})$ a 
$\rho$-tableau.
Since $\mu$-tableaux have one more box than $\rho$-tableaux, 
Vershik \cite[Theorem 4]{V} claims that the removal of one box from 
$\mu$-tableaux gives a bijection from $\mathcal{L}$ to $\mathcal{R}$, where 
$$\cR=\bigcup_{1\leq x\leq h}\SSYT(\rho,\lambda^{(x)})$$
has a natural bijective correspondence with $\cR'$.
Vershik \cite[Section 4]{V} gives examples, 
each of which comes with a bijection. 
However, if $\lambda=(3,3,2)\vdash 8$ and 
$\rho=(4,3)\vdash 7$ then there is no bijection from 
$\mathcal{L}$ to $\mathcal{R}$ arising from the removal of one box (see Example~\ref{vexam.1}).

\begin{exam}[{\cite[Example 1]{V}}]\label{vexam.2}
Let $\lambda=(3,2,1)\vdash 6$ and $\rho=(4,1)\vdash 5$. Then
\begin{align*}
\mu\text{-tableaux : }&A=\begin{tabular}{|c|c|c|c|c|}
\hline $1$ & $1$ & $1$ & $2$ & $2$ \\ \cline {1-5}  $3$\\
\cline{1-1} 
\end{tabular},\;
B=\begin{tabular}{|c|c|c|c|c|}
\hline $1$ & $1$ & $1$ & $2$ & $3$ \\ \cline {1-5}  $2$\\
\cline{1-1} 
\end{tabular},\;
C=\begin{tabular}{|c|c|c|c|}
\hline $1$ & $1$ & $1$ & $2$ \\ \cline {1-4}  $2$& $3$\\
\cline{1-2} 
\end{tabular},
\nnext
D=\begin{tabular}{|c|c|c|c|}
\hline $1$ & $1$ & $1$ & $3$ \\ \cline {1-4}  $2$& $2$\\
\cline{1-2} 
\end{tabular},\;
E=\begin{tabular}{|c|c|c|c|}
\hline $1$ & $1$ & $1$ & $2$ \\ \cline {1-4}  $2$\\
\cline{1-1} $3$\\ \cline{1-1} 
\end{tabular};\\
\rho\text{-tableaux : }&L=\begin{tabular}{|c|c|c|c|}
\hline $1$ & $1$ & $2$ & $2$ \\ \cline {1-4}  $3$\\
\cline{1-1} 
\end{tabular},\;
M=\begin{tabular}{|c|c|c|c|}
\hline $1$ & $1$ & $2$ & $3$ \\ \cline {1-4}  $2$\\
\cline{1-1} 
\end{tabular},\;
N=\begin{tabular}{|c|c|c|c|}
\hline $1$ & $1$ & $1$ & $2$ \\ \cline {1-4}  $3$\\
\cline{1-1} 
\end{tabular},\nnext
P=\begin{tabular}{|c|c|c|c|}
\hline $1$ & $1$ & $1$ & $3$ \\ \cline {1-4}  $2$\\
\cline{1-1} 
\end{tabular},\;
Q=\begin{tabular}{|c|c|c|c|}
\hline $1$ & $1$ & $1$ & $2$ \\ \cline {1-4}  $2$\\
\cline{1-1} 
\end{tabular}.
\end{align*}
We remove one box from the first row in $A$ and $B$, one box from the second row in $C$ and $D$, and one box $(3,1)$ in $E$ in order to obtain $\rho$-tableaux. Then we have a bijection as follows:
\[
A\leftrightarrow L;\quad B\leftrightarrow M;\quad
C\leftrightarrow N;\quad D\leftrightarrow P;\quad E\leftrightarrow Q.
\]
The bijection given by Theorem~\ref{vthm.1} is:
\begin{align*}
L&\leftrightarrow (L\la 1)=E;\quad M\leftrightarrow (M\la 1)=D;\\
N&\leftrightarrow (N\la 2)=A;\quad P\leftrightarrow (P\la 2)=C;\\
Q&\leftrightarrow (Q\la 3)=B.
\end{align*}
\end{exam}

We give an example, for which there is no bijection arising from the removal of one box.

\begin{exam}\label{vexam.1}
Let $\lambda=(3,3,2)\vdash 8$ and $\rho=(4,3)\vdash 7$. Then
\begin{align*}
\mu\text{-tableaux : }&A=\begin{tabular}{|c|c|c|c|c|}
\hline $1$ & $1$ & $1$ & $3$ & $3$ \\ \cline {1-5}  $2$& $2$ & $2$\\
\cline{1-3} 
\end{tabular},\;
B=\begin{tabular}{|c|c|c|c|c|}
\hline $1$ & $1$ & $1$ & $2$ & $3$ \\ \cline {1-5}  $2$& $2$ & $3$\\
\cline{1-3} 
\end{tabular},\;
C=\begin{tabular}{|c|c|c|c|c|}
\hline $1$ & $1$ & $1$ & $2$ & $2$ \\ \cline {1-5}  $2$& $3$ & $3$\\
\cline{1-3} 
\end{tabular},\nnext
D=\begin{tabular}{|c|c|c|c|}
\hline $1$ & $1$ & $1$ & $2$ \\ \cline {1-4}  $2$& $2$ & $3$ & $3$\\
\cline{1-4} 
\end{tabular},\;
E=\begin{tabular}{|c|c|c|c|}
\hline $1$ & $1$ & $1$ & $3$ \\ \cline {1-4}  $2$& $2$ & $2$\\
\cline{1-3} $3$\\ \cline{1-1} 
\end{tabular},\;
F=\begin{tabular}{|c|c|c|c|}
\hline $1$ & $1$ & $1$ & $2$ \\ \cline {1-4}  $2$& $2$ & $3$\\
\cline{1-3} $3$\\ \cline{1-1} 
\end{tabular}
;\\
\rho\text{-tableaux : }
&L=\begin{tabular}{|c|c|c|c|}
\hline $1$ & $1$ & $2$ & $3$ \\ \cline {1-4}  $2$& $2$ & $3$\\
\cline{1-3} 
\end{tabular},\;
M=\begin{tabular}{|c|c|c|c|}
\hline $1$ & $1$ & $2$ & $2$ \\ \cline {1-4}  $2$& $3$ & $3$\\
\cline{1-3} 
\end{tabular},\;
N=\begin{tabular}{|c|c|c|c|}
\hline $1$ & $1$ & $1$ & $3$ \\ \cline {1-4}  $2$& $2$ & $3$\\
\cline{1-3} 
\end{tabular},\nnext
P=\begin{tabular}{|c|c|c|c|}
\hline $1$ & $1$ & $1$ & $2$ \\ \cline {1-4}  $2$& $3$ & $3$\\
\cline{1-3} 
\end{tabular},\;
Q=\begin{tabular}{|c|c|c|c|}
\hline $1$ & $1$ & $1$ & $3$ \\ \cline {1-4}  $2$& $2$ & $2$\\
\cline{1-3} 
\end{tabular},\;
R=\begin{tabular}{|c|c|c|c|}
\hline $1$ & $1$ & $1$ & $2$ \\ \cline {1-4}  $2$& $2$ & $3$\\
\cline{1-3} 
\end{tabular}.
\end{align*}
As mentioned in Section~\ref{vsec.1}, 
$\mu$-tableaux $A$ and $E$ result in $\rho$-tableau $Q$,
so there is no bijection between $\mu$-tableaux and $\rho$-tableaux 
arising from the removal of one box. 
The bijection given by Theorem~\ref{vthm.1} is:
\begin{align*}
L&\leftrightarrow (L\la 1)=E;\quad M\leftrightarrow (M\la 1)=F;\\
N&\leftrightarrow (N\la 2)=D;\quad P\leftrightarrow (P\la 2)=C;\\
Q&\leftrightarrow (Q\la 3)=A;\quad R\leftrightarrow (R\la 3)=B.
\end{align*}
\end{exam}


\begin{thebibliography}{99}
\bibitem{C} T. Ceccherini-Silverstein, F. Scarabotti and F. Tolli, Representation Theory of the Symmetric Groups, Cambridge University Press, 2010.




\bibitem{F} W. Fulton, Young Tableaux, London Mathematical Society
Student Texts 35, 1997.





\bibitem{Ni} Nicholas A. Loehr, Bijective Combinatorics, Chapman and Hall/CRC Press, 2011.


\bibitem{V} A. M. Vershik, A new approach to the representation theory of the symmetric groups, III: Induced representations and Frobenius-Young correspondence, Mosc. Math. J. 6 (2006), no. 3, 567--585.



\end{thebibliography}
\end{document}